\documentclass[10pt]{amsart}
\usepackage{verbatim}
\usepackage{eucal,url,amssymb,stmaryrd,enumerate,amscd}
\usepackage{amsmath}
\usepackage[pagebackref]{hyperref}
\usepackage{amsfonts}
\usepackage{amsthm}
\usepackage[margin=1in]{geometry}%2.2cm

\numberwithin{equation}{section}
\newtheorem{thrm}{Theorem}[section]
\newtheorem{lemma}[thrm]{Lemma}
\newtheorem{prop}[thrm]{Proposition}
\newtheorem{cor}[thrm]{Corollary}
\newtheorem{dfn}[thrm]{Definition}

\newtheorem{rmrk}[thrm]{Remark}

\newtheorem{conv}[thrm]{Convention}

\newcommand{\Om}{\Omega}

\newcommand{\QH}{\boldsymbol {G\,(\mathbb{H})}}

\newcommand{\lap}{\triangle}%{\mathcal{L}}

\newcommand{\norm}[1]{\lVert#1\rVert}

 1

\def\gr{\nabla f}
\def\bi{\nabla}
\newcommand{\HessfI}{\sum_{s=1}^3 \left[ g \left(\nabla^2f , \omega_s \right)\right]^2}
\newcommand{\HessfIs}{g(\nabla^2f , \omega_s)}
\newcommand{\vol}{\, Vol_{\eta}}%=\eta_1\wedge\eta_2\wedge\eta_3\wedge\omega_s^{2n}}

\begin{document}

\begin{abstract}
The main technical result of the paper is a Bochner type formula for the sub-laplacian on a quaternionic contact manifold. With the help of this formula we establish a version of Lichnerowicz' theorem giving a lower bound of the eigenvalues of the sub-Laplacian under a lower bound on the $Sp(n)Sp(1)$ components of the qc-Ricci curvature. It is shown that in the case of a 3-Sasakian manifold the lower bound is reached iff the quaternionic contact manifold is a round 3-Sasakian sphere. Another goal of the paper is to establish a-priori estimates for square integrals of  horizontal derivatives of smooth compactly supported functions. As an application, we prove a sharp inequality bounding the horizontal Hessian of a function by its sub-Laplacian on the quaternionic Heisenberg group.
\end{abstract}



\title[The sharp lower bound of the first eigenvalue]
{The sharp lower bound of the first eigenvalue of the sub-Laplacian on a quaternionic contact manifold}
\date{\today }
\thanks{This work has been ratially funded by}
\author{S. Ivanov}
\address[Stefan Ivanov]{University of Sofia, Faculty of Mathematics and
Informatics, blvd. James Bourchier 5, 1164, Sofia, Bulgaria}
\email{ivanovsp@fmi.uni-sofia.bg}
\author{A. Petkov}
\address[Alexander Petkov]{University of Sofia, Faculty of Mathematics and
Informatics, blvd. James Bourchier 5, 1164, Sofia, Bulgaria}
\email{a\_petkov\_fmi@abv.bg}
\author{D. Vassilev}
\address[Dimiter Vassilev]{ Department of Mathematics and Statistics\\
University of New Mexico\\
Albuquerque, New Mexico, 87131-0001} \email{vassilev@math.unm.edu}
\maketitle \tableofcontents

\setcounter{tocdepth}{2}

\section{Introduction}
  The first circle of results of this paper are motivated by the classical theorems of Lichnerowicz \cite{Li} and Obata \cite{O3} giving correspondingly a lower bound of the first eigenvalue of the Laplacian on a compact manifold with a Ricci  bound and characterizing the case of equality.
%\begin{comment}
%\begin{theorem}\label{t:Lich-Ob}
 In fact,  in \cite{Li} it was shown that for every compact Riemannian manifold of dimension $n$ for which the Ricci curvature is greater than or equal to that of the round unit $n$-dimensional sphere $S^n(1)$, i.e.,
  $$Ric(X,Y)\geq (n-1)g(X,Y)$$
    we have that the first positive eigenvalue $\lambda_1$ of the (positive) Laplace operator  is greater than or equal to the first eigenvalue of the sphere, $$\lambda_1\geq n.$$ Subsequently in \cite{O3} it was shown that the lower bound for the eigenvalue is achieved iff the Riemannian manifold is isometric to $S^n(1)$.
Lichnerowicz proved his result using the classical Bochner-Weitzenb\"ock formula. In turn, Obata showed that under these assumptions the eigenfunction $\phi$ satisfies the system $\nabla^2 \phi = -\phi g$, after which he defines an isometry using analysis based on the geodesics and Hessian comparison of the distance function from a point. Later Gallot \cite{Gallot79} generalized these results to statements involving the higher eigenvalues and corresponding eigenfunctions of the Laplace operator.

 It is natural to ask if there is a sub-Riemannian version of the above results. Greenleaf \cite{Gr} gave a version of Lichnerowicz' result
 %Theorem  \ref{t:Lich-Ob} a) for
 on a compact strongly pseudo-convex CR manifold.
 %\begin{comment}
 %\begin{theorem}[\cite{Gr}]
 Suppose $M$ is $2n+1$, $n\geq 3$ dimensional strongly pseudo-convex CR manifold. If $$Ric(X,Y)+ 4A(X,JY)\geq (n+1) g(X,X)$$ for all horizontal vectors $X$, where $Ric$ and $A$ are, correspondingly, the Ricci curvature and the Webster torsion of the Tanaka-Webster connection (in the notation from \cite{IVZ,IV2}), then the first positive eigenvalue $\lambda_1$ of the  sub-Laplacian satisfies the inequality $\lambda_1\geq  {n}$. The standard CR structure on the sphere achieves equality in this inequality.
 %\end{theorem}
 %\end{comment}
  Further results in the CR case have been proved in \cite{LL}, \cite{Chi06}, \cite{CC07,CC09b,CC09a}, \cite{Bar} and \cite{ChW} adding a corresponding inequality for $n=1$, or characterizing the equality case in the vanishing torsion case  (the Sasakian case).

One purpose of this paper is to consider these questions in the setting of a closed compact quaternionic contact manifold. The Lichnerowicz type result is as follows.
\begin{thrm}\label{main1}
Let $(M,g,\mathbb Q)$ be a compact quaternionic contact manifold of dimension $4n+3>7$. If the Ricci tensor and torsion of the Biquard connection satisfy the inequality
\begin{equation}\label{condm}
Ric(X,X)+\frac{2(4n+5)}{2n+1}T^0(X,X)+ \frac{3(2n^2+5n-1)}{(n-1)(2n+1)}U(X,X)\ge k_0g(X,X)
\end{equation}
for some positive constant $k_0$ then any positive eigenvalue $\lambda$ of the sub-Laplacian $\triangle$ satisfies the inequality
$$\lambda \ge \frac{n}{n+2}k_0.$$
\end{thrm}
The second goal then is to investigate the case of equality in Theorem~\ref{main1}. We restrict our considerations to the case when the torsion of the Biquard connection vanishes, $T^0=U=0$. In this case it is known \cite{IMV} that the qc manifold is qc-Einstein, $Ric=k.g$, the qc-scalar curvature is constant $(n>1)$ and if it is positive then the qc manifold is locally qc equivalent to a 3-Sasakian space. The corresponding result in the negative scalar curvature case can be found in \cite{IV1} and \cite{IV2}. We prove the following result.
\begin{thrm}\label{main0}
Let $(M,g,\mathbb Q)$ be a compact qc-Einstein manifold of dimension $4n+3>7$ of  qc scalar curvature $Scal =16n(n+2)$,
$$Ric(X,Y)=\frac1{4n}Scal g(X,Y)=4(n+2)g(X,Y).
$$
The first positive eigenvalue $\lambda_1$ of the sub-Laplacian equals $4n$
%$$\lambda_1 = \frac{n}{n+2}4(n+2)=4n$$
if and only if  $(M,g,\mathbb Q)$  is qc equivalent to the 3-Sasakian sphere of dimension $4n+3$.

In particular, on a 3-Sasakian manifold of dimension (4n+3), $n>1$, the first positive eigenvalue of the sub-laplacian is equal to $4n$ if and only if the 3-Sasakian manifold is qc-equivalent to the 3-Sasakian sphere.
\end{thrm}
We note that in \cite{IMV2} is given an explicit formula for the eigenfunctions of the above eigenvalue, see also \cite{ACB}.

%\textbf{Conjecture} We conjecture that Theorem~\ref{main0} is valid for any qc manifold of dimension bigger than %seven.

The second main theme of the paper is the  derivation of an a-priori inequality between the (horizontal) Hessian and sub-Laplacian. For the Heisenberg group a corresponding sharp estimate was found in \cite{DoMa05}. Equipped with our estimate we precise the scope of use of \cite{Do08} for the quaternionic Heisenberg group. We recall that the main application is the establishment of the  $\mathcal{C}^{1,\alpha}$ regularity for the $p$ sub-Laplacian with $p$ close to 2. The exact interval for $p$ around 2 is determined by the constant found in this paper. Using Bochner's identity we will find an integral identity. Such integral identities have been exploited earlier in \cite{Gar08}  in the setting of Carnot groups.  A similar method based on Greenleaf's formula was employed in \cite{ChMa10}, but due to the different quaternionic linear algebra our proof proceeds in a way particular to the quaternionic case. Here we prove the  following  result.
\begin{thrm}\label{t:lapl-hess est}
Let $(M, \eta)$ be a $(4n+3)$-dimensional qc manifold, $n>1$. For any $f\in \mathcal{C}_o^\infty (M)$ the following inequality holds true
\begin{multline}\label{e:lapl-hess est}
\int_M \vert \lap f \vert^2\vol \geq \frac {n}{n+1}\int_M \vert \bi^2 f\vert^2\vol
+ \frac {n^2}{n^2-1}\int_M Ric(\gr,\gr) \vol\\
+  \frac {n^2}{n^2-1} \int_M -\frac {4}{n}T^0(\gr, \gr)-6U(\gr,\gr) -6 S|\gr|^2 \vol\\
=\frac {n}{n+1}\int_M \vert \bi^2 f\vert^2\vol
+ \int_M \frac {2n(n+2)}{n+1}T^0(\gr, \gr) + \frac {4n^2}{n-1}U(\gr,\gr) +\frac {2n^2}{n+1}S|\gr|^2 \vol.
\end{multline}
\end{thrm}

The proofs of the last  Theorem is presented in Section \ref{s:cordes}.

\begin{conv}
\label{conven} \hfill\break\vspace{-15pt}
\begin{enumerate}
\item[a)] We shall use $X,Y,Z,U$ to denote horizontal vector fields, i.e.
$X,Y,Z,U\in H$.
\item[b)] $\{e_1,\dots,e_{4n}\}$ denotes a  local orthonormal basis of the horizontal
space $H$.
\item[c)] The summation convention over repeated vectors from the basis $%
\{e_1,\dots,e_{4n}\}$ will be used. For example, for a (0,4)-tensor $P$, the
formula $k=P(e_b,e_a,e_a,e_b)$ means
$$k=\sum_{a,b=1}^{4n}P(e_b,e_a,e_a,e_b);$$
\item[d)] The triple $(i,j,k)$ denotes any cyclic permutation of $(1,2,3)$.
\item[e)] $s$  will be any number from the set $\{1,2,3\}, \quad
s\in\{1,2,3\}$.
\end{enumerate}
\end{conv}

\textbf{Acknowledgments}
The research is  partially supported by the Contract 181/2011
with the University of Sofia `St.Kl.Ohridski' and Contract ``Idei", DID 02-39/21.12.2009. . S.I and D.V. are
partially supported by Contract ``Idei", DO 02-257/18.12.2008.

\section{Quaternionic contact manifolds}

In this section we will briefly review the basic notions of
quaternionic contact geometry and recall some results from
\cite{Biq1}, \cite{IMV} and \cite{IV} which we will use in this
paper.

\subsection{Quaternionic contact structures and the Biquard connection}

A quaternionic contact (qc) manifold $(M, g, \mathbb{Q})$ is a
$4n+3$-dimensional manifold $M$ with a codimension three
distribution $H$  locally given as the kernel of a 1-form
$\eta=(\eta_1,\eta_2,\eta_3)$ with values in $\mathbb{R}^3$. In addition $H$ has
an $Sp(n)Sp(1)$ structure, that is, it is
equipped with a Riemannian metric $g$ and a rank-three bundle
$\mathbb Q$ consisting of
endomorphisms of $H$ locally generated
by three almost complex structures $I_1,I_2,I_3$ on $H$ satisfying
the identities of the imaginary unit quaternions,
$I_1I_2=-I_2I_1=I_3, \quad I_1I_2I_3=-id_{|_H}$ which are
hermitian compatible with the metric $g(I_s.,I_s.)=g(.,.)$ and the following
compatibility condition holds
$\qquad 2g(I_sX,Y)\ =\ d\eta_s(X,Y), \quad X,Y\in H.$

A special phenomena, noted in \cite{Biq1}, is that the contact
form $\eta$ determines the  quaternionic structure and the metric on
the horizontal distribution in a unique way.

If the first Pontryagin class of $M$ vanishes then the 2-sphere bundle of
$\mathbb{R}^3$-valued 1-forms is trivial \cite{AK}, i.e. there is a
globally defined form $\eta$ that anihilates $H$, we denote the
corresponding qc manifold $(M,\eta)$. In this case the 2-sphere of
associated almost complex structures is also globally defined on
$H$.

On a qc manifold with a fixed metric $g$ on $H$ there exists a
canonical connection defined in \cite{Biq1} when the dimension $(4n+3)>7$,
and in \cite{D} for the 7-dimensional case.
\begin{thrm}\label{biqcon}.\cite{Biq1} {Let $(M, g,\mathbb{Q})$ be a qc
manifold} of dimension $4n+3>7$ and a fixed metric $g$ on $H$ in
the conformal class $[g]$. Then there exists a unique connection
$\nabla$ with
torsion $T$ on $M^{4n+3}$ and a unique supplementary subspace $V$ to $H$ in
$TM$, such that:
\begin{enumerate}
\item[i)]
$\nabla$ preserves the decomposition $H\oplus V$ and the $Sp(n)Sp(1)$ structure on $H$,
i.e. $\nabla g=0,  \nabla\sigma \in\Gamma(\mathbb Q)$ for a section
$\sigma\in\Gamma(\mathbb Q)$, and its torsion on $H$ is given by $T(X,Y)=-[X,Y]_{|V}$;
\item[ii)] for $\xi\in V$, the endomorphism $T(\xi,.)_{|H}$ of $H$ lies in
$(sp(n)\oplus sp(1))^{\bot}\subset gl(4n)$;
\item[iii)]
the connection on $V$ is induced by the natural identification $\varphi$ of
$V$ with the subspace $sp(1)$ of the endomorphisms of $H$, i.e.
$\nabla\varphi=0$.
\end{enumerate}
\end{thrm}
In ii), the inner product $<,>$ of $End(H)$ is given by $<A,B> = {
\sum_{i=1}^{4n} g(A(e_i),B(e_i)),}$ for $A, B \in End(H)$. We shall call the above connection \emph{the Biquard connection}.
When the dimension of $M$ is at least eleven \cite{Biq1} also described the supplementary distribution $V$, which is (locally) generated by the so called Reeb vector fields  $\{\xi_1,\xi_2,\xi_3\}$ determined by
\begin{equation}  \label{bi1}
\begin{aligned} \eta_s(\xi_k)=\delta_{sk}, \qquad (\xi_s\lrcorner
d\eta_s)_{|H}=0,\\ (\xi_s\lrcorner d\eta_k)_{|H}=-(\xi_k\lrcorner
d\eta_s)_{|H}, \end{aligned}
\end{equation}
where $\lrcorner$ denotes the interior multiplication.  If the dimension of $M$ is seven Duchemin shows in \cite{D} that if we assume, in addition, the existence of Reeb vector fields as in \eqref{bi1}, then Theorem~\ref{biqcon} holds. Henceforth, by a qc structure in dimension $7$ we shall mean a qc structure satisfying \eqref{bi1}.

The qc conformal curvature tensor $W^{qc}$, introduced  in \cite{IV},
is the obstruction for a qc structure to be locally qc
conformal to the flat structure on the quaternionic Heisenberg group $
\boldsymbol{G\,(\mathbb{H})}$.
A qc conformally flat structure is also locally qc conformal to
the standard 3-Sasaki sphere due to the local qc conformal equivalence of
the standard 3-Sasakian structure on the $(4n+3)$-dimensional sphere
and the quaternionic Heisenberg group \cite{IMV,IV}.

Notice that equations \eqref{bi1} are invariant under the natural
$SO(3)$
action. Using the triple of Reeb vector fields we extend $g$ to a metric on
$M$ by requiring
$span\{\xi_1,\xi_2,\xi_3\}=V\perp H \text{ and }
g(\xi_s,\xi_k)=\delta_{sk}.
$
\hspace{2mm} \noindent The extended metric does not depend on the action of $SO(3)$ on $V$, but it changes in an obvious manner if $\eta$ is multiplied by a conformal factor. Clearly, the Biquard connection
preserves the extended metric on $TM, \nabla g=0$. Since the Biquard connection is metric it is connected with the Levi-Civita connection $\nabla^g$ of the metric $g$ by the general formula
\begin{equation}  \label{lcbi}
g(\nabla_AB,C)=g(\nabla^g_AB,C)+\frac12\Big[
g(T(A,B),C)-g(T(B,C),A)+g(T(C,A),B)\Big].
\end{equation}

The covariant derivative of the qc structure with respect to the Biquard connection and the covariant derivative of the distribution $V $ are given by
\begin{equation}\label{xider}
\nabla I_i=-\alpha_j\otimes I_k+\alpha_k\otimes I_j,\quad
\nabla\xi_i=-\alpha_j\otimes\xi_k+\alpha_k\otimes\xi_j.
\end{equation}
The $sp(1)$-connection 1-forms $\alpha_s$ on $H$ are expressed in
\cite{Biq1} by
\begin{gather}  \label{coneforms}
\alpha_i(X)=d\eta_k(\xi_j,X)=-d\eta_j(\xi_k,X), \quad X\in H, \quad
\xi_i\in V,
\end{gather}
while the $sp(1)$-connection 1-forms $\alpha_s$ on the vertical
space $V$ are calculated in \cite{IMV}
\begin{gather}  \label{coneform1}
\alpha_i(\xi_s)\ =\ d\eta_s(\xi_j,\xi_k) -\
\delta_{is}\left(\frac{S}2\ +\ \frac12\,\left(\,
d\eta_1(\xi_2,\xi_3)\ +\ d\eta_2(\xi_3,\xi_1)\ + \
d\eta_3(\xi_1,\xi_2)\right)\right),
\end{gather}
where  $S$ is the \emph{normalized} qc scalar curvature defined
below in \eqref{qscs}. The vanishing of the $sp(1)$-connection
$1$-forms on $H$ implies the vanishing of the torsion endomorphism of
the Biquard connection (see \cite{IMV}).

The fundamental 2-forms $\omega_s$
of the quaternionic structure $Q$ are defined by
\begin{equation}  \label{thirteen}
2\omega_{s|H}\ =\ \, d\eta_{s|H},\qquad \xi\lrcorner\omega_s=0,\quad \xi\in
V.
\end{equation}
Due to \eqref{thirteen}, the torsion restricted to $H$ has the form
\begin{equation}  \label{torha}
T(X,Y)=-[X,Y]_{|V}=2\omega_1(X,Y)\xi_1+2\omega_2(X,Y)\xi_2+2\omega_3(X,Y)\xi_3.
\end{equation}

\subsection{Invariant decompositions}
Any endomorphism $\Psi$ of $H$ can be  decomposed with respect to the quaternionic
structure $(\mathbb{Q},g)$ uniquely into four $Sp(n)$-invariant parts
%\begin{equation}\label{New4}
$\Psi=\Psi^{+++}+\Psi^{+--}+\Psi^{-+-}+\Psi^{--+},$
%\end{equation}
where $\Psi^{+++}$ commutes with all three $I_i$, $\Psi^{+--}$ commutes with $I_1$ and
anti-commutes with the others two and etc. Explicitly,
\begin{equation}\label{spsp}
\begin{aligned}4\Psi^{+++}=\Psi-I_1\Psi I_1-I_2\Psi I_2-I_3\Psi
I_3,\qquad
4\Psi^{+--}=\Psi-I_1\Psi I_1+I_2\Psi I_2+I_3\Psi I_3,\\
4\Psi^{-+-}=\Psi+I_1\Psi I_1-I_2\Psi I_2+I_3\Psi I_3,\qquad 4\Psi^{--+}=\Psi+I_1\Psi
I_1+I_2\Psi I_2-I_3\Psi I_3.
\end{aligned}
\end{equation}
\noindent The two $Sp(n)Sp(1)$-invariant components \index{$Sp(n)Sp(1)$-invariant components!$\Psi_{[3]}$} \index{$Sp(n)Sp(1)$-invariant components!$\Psi_{[-1]}$} are given by
$$\Psi_{[3]}=\Psi^{+++},\quad \Psi_{[-1]}=\Psi^{+--}+\Psi^{-+-}+\Psi^{--+}
$$
with the following characterizing equations
\begin{equation}{\label{New21}}
\begin{aligned}
\Psi=\Psi_{[3]} \quad \Longleftrightarrow 3\Psi+I_1\Psi I_1+I_2\Psi I_2+I_3\Psi I_3=0,\\
\Psi=\Psi_{[-1]}\quad \Longleftrightarrow \Psi-I_1\Psi I_1-I_2\Psi I_2-I_3\Psi I_3=0.
\end{aligned}
\end{equation}
\noindent With a short calculation one sees
that the $Sp(n)Sp(1)$-invariant components are the projections on the eigenspaces of the
Casimir operator
\begin{equation}\label{e:cross}
\Upsilon =\ I_1\otimes I_1\ +\ I_2\otimes I_2\ +\ I_3\otimes I_3
\end{equation}
corresponding, respectively, to the eigenvalues $3$ and $-1$, see
\cite{CSal}. If $n=1$ then the space of symmetric endomorphisms
commuting with all $I_s$ is 1-dimensional, i.e. the
[3]-component of any symmetric endomorphism $\Psi$ on $H$ is
proportional to the identity,
$\Psi_{3}=\frac{|\Psi|^2}{4}Id_{|H}$. Note here that each of the
three 2-forms $\omega_s$ belongs to its [-1]-component,
$\omega_s=\omega_{s[-1]}$ and constitute a basis of the lie
algebra $sp(1)$.

\subsection{The torsion tensor}
The properties of the Biquard connection are encoded in the
properties of
the torsion
endomorphism $T_{\xi}=T(\xi,\cdot) : H\rightarrow H, \quad \xi\in V$. Decomposing
the endomorphism $T_{\xi}\in(sp(n)+sp(1))^{\perp}$
into its symmetric part $T^0_{\xi}$ and skew-symmetric part
$b_{\xi}, T_{\xi}=T^0_{\xi} + b_{\xi} $, O. Biquard shows in
\cite{Biq1} that the torsion $T_{\xi}$ is completely trace-free,
$tr\, T_{\xi}=tr\, T_{\xi}\circ
I_s=0$, its symmetric part has the properties
$T^0_{\xi_i}I_i=-I_iT^0_{\xi_i}\quad
I_2(T^0_{\xi_2})^{+--}=I_1(T^0_{\xi_1})^{-+-},\quad
I_3(T^0_{\xi_3})^{-+-}=I_2(T^0_{\xi_2})^{--+},\quad
I_1(T^0_{\xi_1})^{--+}=I_3(T^0_{\xi_3})^{+--} $,
where the upperscript $+++$ means commuting  with all three $I_i$, $+--$
indicates commuting with $I_1$ and anti-commuting with the other two and etc.
The skew-symmetric part can be represented as $b_{\xi_i}=I_iu$, where
$u$ is a traceless symmetric (1,1)-tensor on $H$ which commutes with
$I_1,I_2,I_3$. If $n=1$ then the tensor $u$ vanishes identically,
$u=0$ and the torsion is a symmetric tensor, $T_{\xi}=T^0_{\xi}$.

Any 3-Sasakian manifold has zero torsion endomorphism, and
 the converse is true if in addition the qc scalar curvature (see
\eqref{qscs}) is a positive constant \cite{IMV}. We remind that a
$(4n+3)$-dimensional  Riemannian manifold $(M,g)$ is
called 3-Sasakian if the cone
metric $g_c=t^2g+dt^2$ on $C=M\times \mathbb{R}^+$ is a
hyper K\"ahler metric, namely, it has holonomy contained
in $Sp(n+1)$ \cite{BGN}. A 3-Sasakian manifold of dimension
$(4n+3)$ is Einstein with positive Riemannian scalar curvature
$(4n+2)(4n+3)$ \cite{Kas} and if complete it is compact with a
finite fundamental group, (see \cite{BG} for a nice overview of
3-Sasakian spaces).

\subsection{Torsion and curvature}

Let $R=[\nabla,\nabla]-\nabla_{[\ ,\ ]}$ be the curvature tensor of
$\nabla$ and the dimension is $4n+3$. We denote the curvature tensor
of type (0,4) and the torsion tensor of type (0,3) by the same letter, $R(A,B,C,D):=g(R(A,B)C,D),\quad T(A,B,C):=g(T(A,B),C)$,
$A,B,C,D \in \Gamma(TM)$. The  Ricci tensor, the normalized
scalar curvature and the  Ricci $2$-forms  of the Biquard connection, called \emph{qc-Ricci tensor} $Ric$,
\emph{normalized qc-scalar curvature} $S$ and \emph{qc-Ricci forms} $\rho_s, \tau_s$, respectively, are
defined by
\begin{equation}  \label{qscs}
\begin{aligned}
Ric(A,B)=R(e_b,A,B,e_b),\quad 8n(n+2)S=R(e_b,e_a,e_a,e_b),\\
\rho_s(A,B)=\frac1{4n}R(A,B,e_a,I_se_a), \quad \tau_s(A,B)=\frac1{4n}R(e_a,I_se_a,A,B,).
\end{aligned}
\end{equation}
The $sp(1)$-part of $R$ is determined by the Ricci 2-forms and the
connection 1-forms by
\begin{equation}  \label{sp1curv}
R(A,B,\xi_i,\xi_j)=2\rho_k(A,B)=(d\alpha_k+\alpha_i\wedge\alpha_j)(A,B),
\qquad A,B \in \Gamma(TM).
\end{equation}
The two $Sp(n)Sp(1)$-invariant trace-free symmetric 2-tensors
$T^0(X,Y)=
g((T_{\xi_1}^{0}I_1+T_{\xi_2}^{0}I_2+T_{ \xi_3}^{0}I_3)X,Y)$, $U(X,Y)
=g(uX,Y)$ on $H$, introduced in \cite{IMV}, have the properties:
\begin{equation}  \label{propt}
\begin{aligned} T^0(X,Y)+T^0(I_1X,I_1Y)+T^0(I_2X,I_2Y)+T^0(I_3X,I_3Y)=0, \\
U(X,Y)=U(I_1X,I_1Y)=U(I_2X,I_2Y)=U(I_3X,I_3Y). \end{aligned}
\end{equation}
In dimension seven $(n=1)$, the tensor $U$ vanishes identically,
$U=0$.

We shall need the following identity taken from
\cite[Proposition~2.3]{IV}
\begin{equation}  \label{need}
4T^0(\xi_s,I_sX,Y)=T^0(X,Y)-T^0(I_sX,I_sY).
\end{equation}
Thus, taking into account \eqref{need} we have the formula
\begin{equation}  \label{need1}
T(\xi_s,I_sX,Y)=T^0(\xi_s,I_sX,Y)+g(I_suI_sX,Y)=\frac14\Big[T^0(X,Y)-T^0(I_sX,I_sY)\Big]-U(X,Y).
\end{equation}
\begin{dfn}  A qc structure is said to be qc Einstein if the horizontal qc-Ricci tensor is a scalar multiple of the metric,
$$Ric(X,Y)=2(n+2)Sg(X,Y).$$
\end{dfn}
The horizontal Ricci tensor and the horizontal Ricci 2-forms can be expressed in terms of the
torsion of the Biquard connection \cite{IMV} (see also
\cite{IMV1,IV}). We collect the necessary facts from
\cite[Theorem~1.3, Theorem~3.12, Corollary~3.14, Proposition~4.3 and
Proposition~4.4]{IMV} with slight modification presented in
\cite{IV}

\begin{thrm}[\cite{IMV}]\label{sixtyseven} On a $(4n+3)$-dimensional qc manifold $(M,\eta,\mathbb{Q})$ with a normalized scalar curvature $S$ we have the following relations
\begin{equation} \label{sixtyfour}
\begin{aligned}
& Ric(X,Y)  =(2n+2)T^0(X,Y)+(4n+10)U(X,Y)+2(n+2)Sg(X,Y),\\
& \rho_s(X,I_sY) = -\frac12\Bigl[T^0(X,Y)+T^0(I_sX,I_sY)\Bigr]-2U(X,Y)-Sg(X,Y),\\
& \tau_s(X,I_sY)  = -\frac{n+2}{2n}\Bigl[T^0(X,Y)+T^0(I_sX,I_sY)\Bigr]-Sg(X,Y),\\
& T(\xi_{i},\xi_{j}) =-S\xi_{k}-[\xi_{i},\xi_{j}]_{H}, \qquad S  = -g(T(\xi_1,\xi_2),\xi_3),\\
& g(T(\xi_i,\xi_j),X) =-\rho_k(I_iX,\xi_i)=-\rho_k(I_jX,\xi_j)=-g([\xi_i,\xi_j],X),\\
& \frac12\xi_j(S) = \rho_i(\xi_i,\xi_j)+\rho_k(\xi_k,\xi_j),\\
& \rho_{i}(\xi_{i},X)  = \frac {X(S)}{4}  + \frac 12\, \left (\rho_{i}(\xi_{j},I_{k}X)-\rho_{j}(\xi_{k},I_{i}X)-\rho_{k}(\xi_{i},I_{j}X)
\right).
 \end{aligned}
\end{equation}
For $n=1$ the above formulas hold with $U=0$.

The qc Einstein condition
is equivalent to the vanishing of the torsion endomorphism of the
Biquard connection. In this case $S$ is constant and  the vertical distribution
is integrale provided $n>1$.
\end{thrm}

\subsection{The Ricci identities}
We shall use repeatedly the following Ricci identities of order two and three, see  also \cite{IV}.
%\begin{prop}\label{p:ricci identities}
Let $\xi_i$, $i=1,2,3$ be  the Reeb vector fields, $X, Y\in H$ and $f$ a smooth function on the qc manifold $M$ with $\bi f$ its horizontal gradient of $f$, $g(\gr,X)=df(X)$.  We have:
\begin{equation}\label{boh2}
\begin{aligned}
& \nabla^2f (X,Y)-\nabla^2f(Y,X)=-2\sum_{s=1}^3\omega_s(X,Y)df(\xi_s)\\
& \nabla^2f (X,\xi_s)-\nabla^2f(\xi_s,X)=T(\xi_s,X,\gr)\\
& \nabla^3 f (X,Y,Z)-\nabla^3 f(Y,X,Z)=-R(X,Y,Z,\gr) - 2\sum_{s=1}^3 \omega_s(X,Y)\nabla^2f (\xi_s,Z)\\
%& \nabla^3 f (X,Y,\xi_i)-\nabla^3 f(Y,X,\xi_i)=-2df(\xi_j)\rho_k(X,Y)-2df(\xi_k)\rho_j(X,Y)\\
%&\hskip1.5truein -2\omega_i(X,Y)\nabla^2f (\xi_i, \xi_i)-2\omega_j(X,Y)\nabla^2f (\xi_j, \xi_i)-2\omega_k(X,Y)\nabla^2f %(\xi_k, \xi_i).
\end{aligned}
\end{equation}
%\end{prop}
where we used \eqref{torha} in the last equalities in the first and the fourth lines.

%\subsection{The qc conformal curvature}

\subsection{The horizontal divergence theorem}
Let $(M, g,\mathbb{Q})$ be a qc manifold of dimension $4n+3>7$
For a fixed local 1-form $\eta$ and a fix $s\in \{1,2,3\}$ the form
\begin{equation}\label{e:volumeform}
Vol_{\eta}=\eta_1\wedge\eta_2\wedge\eta_3\wedge\omega_s^{2n}
\end{equation}
is a locally defined volume form. Note that $Vol_{\eta}$ is independent
on $s$ as well as it is independent on the local one forms $\eta_1,\eta_2,\eta_3$. Hence it is globally defined volume form denoted with $\vol$.

We consider the (horizontal) divergence of a horizontal vector
field/one-form $\sigma\in\Lambda^1\, (H)$ defined by
\begin{equation}\label{e:divergence}
\nabla^*\, \sigma\ =-tr|_{H}\nabla\sigma=\
-\nabla \sigma(e_a,e_a).
\end{equation}
We need  the following Proposition from \cite{IMV}, see also \cite{Wei}, which allows "integration by parts".
\begin{prop}[\cite{IMV}]\label{div}
On a compact  quaternionic contact manifold $(M,\eta)$  the following divergence formula holds true
\[
\int_M (\nabla^*\sigma)\,\vol\ =\ 0.
\]
\end{prop}

\section{The Bochner formula for the sub-Laplacian}

The horizontal sub-Laplacian $\triangle f$ and the norm of the horizontal gradient $\nabla f$ of a smooth function $f$ on $M$ are defined respectively by
\begin{equation}\label{lap}\triangle f\ =-\ tr^g_H(\nabla^2f)\ =\bi^*df= -\
\nabla^2f(e_a,e_a), \qquad |\nabla f|^2\ =\
df(e_a)\,df(e_a).
\end{equation}
The function $f$ is an eigenfunction with eigenvalue $\lambda$ of the sub-Laplacian if
\begin{equation}\label{eig}
\triangle f =\lambda f,
\end{equation}
for some constant $\lambda$. The divergence formula implies that on a compact qc manifolds all eigenvalues of the sub-Laplacian are non-negative. Our main result Theorem \ref{main1} gives a lower bound on the positive eigenvalues. Therefore, Theorem \ref{main1} can also be interpreted as giving a bound from below on the  first eigenvalue $\lambda_1$, i.e., of the smallest positive eigenvalue for which \eqref{eig} holds.

We start with the proof of the following Bochner-type formula.
\begin{thrm}\label{t:bohh}
On a qc manifold of dimension $4n+3$ the  next formula holds true
\begin{multline}\label{bohh}
\frac12\triangle |\nabla f|^2=|\bi^2f|^2-g\left (\nabla (\triangle f), \gr \right )+Ric(\nabla f,\nabla f)+2\sum_{s=1}^3T(\xi_s,I_s\gr,\gr)+
4\sum_{s=1}^3\bi^2f(\xi_s,I_s\gr).
\end{multline}
\end{thrm}
\begin{proof}
By definition we have
\begin{equation}\label{boh1}
-\frac12\triangle |\nabla f|^2=\bi^3f(e_a,e_a,e_b) df(e_b)+\bi^2f(e_a,e_b)\bi^2f(e_a,e_b)
=\bi^3f(e_a,e_a,e_b) df(e_b) + |\bi^2f|^2.
\end{equation}
To evaluate the first term in the right hand side of \eqref{boh1} we use the Ricci identities \eqref{boh2}. An application of \eqref{xider} to \eqref{torha} gives
\begin{equation}\label{par}
(\bi_X T)(Y,Z)=0.
\end{equation}
Applying successively the Ricci identities \eqref{boh2} and also \eqref{par} we obtain the next sequence of equalities
\begin{multline}\label{boh3}
\bi^3f(e_a,e_a,e_b)df(e_b)=\bi^3f(e_a,e_b,e_a) df(e_b)-2\sum_{s=1}^3\omega_s(e_a,e_b)df(e_b)\bi^2f(e_a,\xi_s)\\=\bi^3f(e_b,e_a,e_a)df(e_b)
-R(e_a,e_b,e_a,e_c)df(e_c)df(e_b)-
2\sum_{s=1}^3\omega_s(e_a,e_b)df(e_b)\Big[\bi^2f(\xi_s,e_a)+\bi^2f(e_a,\xi_s)\Big]\\= -d(\triangle f)(e_b) df(e_b)+Ric(\nabla f,\nabla f)+
4\sum_{s=1}^3\bi^2f(\xi_s,I_s\gr)+2\sum_{s=1}^3T(\xi_s,I_s\gr,\gr)
\end{multline}
A substitution of \eqref{boh3}  in \eqref{boh1} completes the proof of \eqref{bohh}.
\end{proof}

\begin{cor}
On a qc manifold of dimension $4n+3$ the  next formula holds
\begin{multline}\label{boh}
\frac12\triangle |\nabla f|^2=-d(\triangle f)(e_a) df(e_a)+Ric(\nabla f,\nabla f)+
2T^0(\nabla f,\gr)-6U(\gr,\gr) +|\bi^2f|^2\\+
4\sum_{s=1}^3\bi^2f(\xi_s,I_s\gr).
\end{multline}
\end{cor}
\begin{proof}
Using \eqref{need1} together with \eqref{propt}, we calculate
\begin{equation}\label{boh4}
2\sum_{s=1}^3T(\xi_s,I_s\gr,\gr)=2T^0(\gr,\gr)-6U(\gr,\gr),
\end{equation}
which when combined with \eqref{boh4} and \eqref{bohh} give \eqref{boh}.
\end{proof}
 Our next goal is  to evaluate in two ways the last term of \eqref{boh}. First using the $Sp(n)Sp(1)$-invariant orthogonal decomposition  $\Psi_{[3]}\oplus\Psi_{[-1]}$ of all linear maps on $H$, we obtain  the $Sp(n)Sp(1)$-invariant  decomposition of the horizontal Hessian $\bi^2f$ (after the usual identification of tensors through the metric), namely
\begin{equation}\label{comp}
\begin{aligned}
(\bi^2f)_{[3]}(X,Y)=\frac14\Big[\bi^2f(X,Y)+\sum_{s=1}^3\bi^2f(I_sX,I_sY)\Big]\\
(\bi^2f)_{[-1]}(X,Y)=\frac14\Big[3\bi^2f(X,Y)-\sum_{s=1}^3\bi^2f(I_sX,I_sY)\Big].
\end{aligned}
\end{equation}
We continue with the next lemma where we give the first formula for the last term of \eqref{boh}.
\begin{lemma}\label{gr1}
On a compact qc manifold of dimension $4n+3$ the next integral formula holds
\begin{equation}\label{1}
\int_M\sum_{s=1}^3\bi^2f(\xi_s,I_s\gr)\vol=\int_M\Big[\frac3{4n}|(\bi^2f)_{[3]}|^2-\frac1{4n}|(\bi^2f)_{[-1]}|^2-\frac12\sum_{s=1}^3\tau_{s}(I_s\gr,\gr)\Big]\vol.
\end{equation}
\end{lemma}
\begin{proof}
 We recall that an orthonormal frame\newline
\centerline{$\{e_1,e_2=I_1e_1,e_3=I_2e_1,e_4=I_3e_1,\dots,
e_{4n}=I_3e_{4n-3}, \xi_1, \xi_2, \xi_3 \}$}\newline is a
qc-normal frame (at a point) if the connection 1-forms of the
Biquard connection vanish (at that point). Lemma~4.5 in \cite{IMV}
asserts that a qc-normal frame exists at each point of a qc
manifold.

Using the identification of the 3-dimensional vector spaces
spanned by $\{\xi_1,\xi_2, \xi_3\} $ and $\{I_1,I_2,I_3 \}$ with
$\mathbb{R}^3$, the restriction of the action of $Sp(n)Sp(1)$ to
this spaces can be identified with  the action of the group
$SO(3)$, i.e.,  $\xi_i = \sum_{t=1}^3 \Psi_{it} \bar \xi_t $ and  $I_i = \sum_{t=1}^3 \Psi_{it} \bar I_t $, $i=1,2,3$ with $\Psi\in SO(3)$. One verifies easily that  the horizontal 1-form
$$B(X)=\sum_{s=1}^3\bi^2f(I_sX,I_se_a)df(e_a)$$
 is $Sp(n)Sp(1)$ invariant on $\mathbb H$, for example $\bar B(X) = (det \Psi)\, B(X)=B(X)$.
Thus, it is sufficient to compute the divergence of $B$  in a qc-normal frame.  To avoid the introduction of new
variables we shall assume that $\{e_1, \dots,
e_{4n}, \xi_1, \xi_2, \xi_3 \}$ is a qc-normal frame.

Using that the Biquard connection preserves the
splitting of $TM$, the Ricci identities \eqref{boh2}, the definition of $\tau_s$ and \eqref{torha}, we find
\begin{multline}\label{div0}
\bi^*B=\sum_{s=1}^3\Big[\bi^3f(e_b,I_se_b,I_se_a) df(e_a)+\bi^2f(I_se_b,I_se_a)\bi^2f(e_b,e_a)\Big]\\=
\frac12\sum_{s=1}^3\Big[\bi^3f(e_b,I_se_b,I_se_a)-\bi^3f(I_se_b,e_b,I_se_a)\Big]df(e_a)+
\sum_{s=1}^3\bi^2f(I_se_b,I_se_a)\bi^2f(e_b,e_a)\\=-\frac12R(e_b,I_se_b,I_se_a,e_c)df(e_c)df(e_a)-\sum_{s=1}^3\omega_s(e_b,I_se_b)\bi^2f(\xi_s,I_se_a)df(e_a)+\sum_{s=1}^3\bi^2f(I_se_b,I_se_a)\bi^2f(e_b,e_a)\\=-2n\sum_{s=1}^3\tau_{s}(I_s\gr,\gr)-4n\sum_{s=1}^3\bi^2f(\xi_s,I_s\gr) + g\left (\Upsilon \bi^2 f, \bi^2 f \right ),%\sum_{s=1}^3\bi^2f(I_se_b,I_se_a)\bi^2f(e_b,e_a),
\end{multline}
where we used \eqref{e:cross} in the last term and the convention $I_s\alpha(X)=-\alpha(I_sX)$ for a horizontal 1-form $\alpha$. Using the orthogonality of the spaces $\Psi_{[3]}$ and $\Psi_{[-1]}$ we have
\[
g\left (\Upsilon \bi^2 f, \bi^2 f \right )= 3|(\bi^2f)_{[3]}|^2- |(\bi^2f)_{[-1]}|^2.
\]
A substitution of the last equality in \eqref{div0} and the divergence formula give \eqref{1}. This completes the proof of the Lemma.
\end{proof}
The second integral formula for the last term in \eqref{boh} follows.
\begin{lemma}\label{gr2}
On a compact qc manifold of dimension $4n+3$ the following integral formula holds
\begin{equation}\label{2}
\int_M\sum_{s=1}^3\bi^2f(\xi_s,I_s\gr)\vol=-\int_M\Big[4n\sum_{s=1}^3(df(\xi_s))^2
+\sum_{s=1}^3T(\xi_s,I_s\gr,\gr)\Big] \vol.
\end{equation}
\end{lemma}
\begin{proof}
Note, that by definition we have $$\Big [ g \left(\nabla^2f , \omega_s \right)\Big ]^2=\Big[\bi^2f(e_a,I_se_a)\Big]^2.$$ From the Ricci identities we have
\begin{equation}\label{xi1}
\HessfIs =\bi^2f(e_a,I_se_a)=-4ndf(\xi_s)
\end{equation}
which implies
\begin{equation}\label{vert1}
16n^2\int_M\sum_{s=1}^3 \left ( df(\xi_s)\right )^2\,\vol=\int_M\HessfI\,\vol\\=-4n\int_M\sum_{s=1}^3\HessfIs {\,} df(\xi_s)\,\vol.
\end{equation}
Let us consider the $Sp(n)Sp(1)$ invariant horizontal 1-form defined by
$$C(X)=\sum_{s=1}^3df(I_sX)df(\xi_s)$$
whose divergence is (computing as usual in a qc normal frame)
\begin{multline}\label{vert2}
\bi^*C=\sum_{s=1}^3\Big[\bi^2f(e_a,I_se_a){\,} df(\xi_s)+\bi^2f(e_a,\xi_s){\,} df(I_se_a)\Big]\\=
\sum_{s=1}^3\Big[\HessfIs{\,} df(\xi_s)-\bi^2f(\xi_s,I_s\gr)-T(\xi_s,I_s\gr,\gr)\Big].
\end{multline}
In the above calculation we used the second formula of \eqref{boh2} to obtain the second equality of \eqref{vert2}. Integrate \eqref{vert2} over $M$ and use \eqref{vert1} to get \eqref{2} which completes the proof of the lemma.
\end{proof}

\section{Proof of Theorem~\ref{main1}}
\begin{proof}
We begin by integrating the Bochner type formula \eqref{bohh} over the compact qc manifold $M$ of dimension $4n+3$. Using the divergence formula we come to
\begin{multline}\label{bohin}
0=\int_M\Big[ -(\triangle f)^2+|(\bi^2f)_{[3]}|^2+|(\bi^2f)_{[-1]}|^2+Ric(\nabla f,\nabla f)+2\sum_{s=1}^3T(\xi_s,I_s\gr,\gr)\Big]\vol\\+
4\int_M\sum_{s=1}^3\bi^2f(\xi_s,I_s\gr)\,\vol.
\end{multline}
Following Greenleaf \cite{Gr}, we represent the last term in \eqref{bohin} as follows
$$ \int_M\sum_{s=1}^3\bi^2f(\xi_s,I_s\gr)\,\vol= (1-c)\int_M\sum_{s=1}^3\bi^2f(\xi_s,I_s\gr)\,\vol+c\int_M\sum_{s=1}^3\bi^2f(\xi_s,I_s\gr)\,\vol,$$
where $c$ is a constant. Then we apply  Lemma~\ref{gr1} and Lemma~\ref{gr2}, correspondingly, to the first and the second terms in the obtained identity  after which the above equality \eqref{bohin} takes the form
\begin{multline}\label{bohin1}
0=\int_M\Big[ -(\triangle f)^2+|(\bi^2f)_{[3]}|^2+|(\bi^2f)_{[-1]}|^2+Ric(\nabla f,\nabla f)+2\sum_{s=1}^3T(\xi_s,I_s\gr,\gr)\Big] \vol\\
+4(1-c)\int_M \Big[ \frac3{4n}|(\bi^2f)_{[3]}|^2-\frac1{4n}|(\bi^2f)_{[-1]}|^2-\frac12\sum_{s=1}^3\tau_{s}(I_s\gr,\gr) \Big]\vol\\
-4c\int_M \Big[4n\sum_{s=1}^3(df(\xi_s))^2
+\sum_{s=1}^3T(\xi_s,I_s\gr,\gr)\Big] \vol.
\end{multline}
Equation  \eqref{bohin1} can be simplified as follows
\begin{multline}\label{bohin2}
0=\int_M\Big[ -(\triangle f)^2+\left (1+\frac{3(1-c)}n\right)|(\bi^2f)_{[3]}|^2+\left(1-\frac{(1-c)}n\right)|(\bi^2f)_{[-1]}|^2 + Ric(\nabla f,\nabla f)
\Big]\vol
\\
+\int_M \Big[-16nc\sum_{s=1}^3(df(\xi_s))^2
-2(1-c)\sum_{s=1}^3\tau_{s}(I_s\gr,\gr)+(2-4c)\sum_{s=1}^3T(\xi_s,I_s\gr,\gr)\Big]\vol.
\end{multline}
%Using the Cauchy-Schwarz inequality, we get from \eqref{comp} that
Using that $\left \{ \frac {1}{2\sqrt n} \omega_s \right \}$ is an orthonormal set in $\Psi_{[-1]}$ we have
\begin{equation}\label{coshy1}
|(\bi^2f)_{[-1]}|^2 \geq \frac1{4n}\HessfI=4n\sum_{s=1}^3(df(\xi_s))^2,
\end{equation}
while a projection on $\left \{ \frac {1}{2\sqrt n} g \right \}$ gives
\begin{equation}\label{coshy3}
|(\bi^2f)_{[3]}|^2 \geq \frac1{4n}(\triangle f)^2.
\end{equation}
We obtain from \eqref{bohin2} taking into account \eqref{coshy3} and \eqref{coshy1} that for any constant $c$ such that \begin{equation}\label{const}
1+\frac{3(1-c)}n \ge 0, \qquad 1-\frac{(1-c)}n\ge 0
\end{equation}
we have the following inequality
\begin{multline}\label{bohin3}
0 \ge \int_M \Big[\Big(\frac1{4n}+\frac{3(1-c)}{4n^2}-1\Big)(\triangle f)^2+16n^2\Big(\frac1{4n}-\frac{(1-c)}{4n^2}-\frac{c}n\Big)\sum_{s=1}^3(df(\xi_s))^2 \Big]\vol
\\+\int_M \Big[Ric(\nabla f,\nabla f)
-2(1-c)\sum_{s=1}^3\tau_{s}(I_s\gr,\gr)+(2-4c)\sum_{s=1}^3T(\xi_s,I_s\gr,\gr)\Big]\vol.
\end{multline}
The coefficient in front of  $\sum_{s=1}^3(df(\xi_s))^2$ is non-negative provided $c\leq\frac{n-1}{4n-1}$, so in order to cancel the $\sum_{s=1}^3(df(\xi_s))^2$  term we take
\begin{equation}\label{c}
c=\frac{n-1}{4n-1}.
\end{equation}
Note that the inequalities \eqref{const} are satisfies and with this choice of  $c$ \eqref{bohin3} yields
\begin{multline}\label{bohin4}
0 \ge \int_M \Big(\frac{2(1-n)(2n+1)}{n(4n-1)} \Big)(\triangle f)^2 \vol
\\+\int_M \Big[Ric(\nabla f,\nabla f) -\frac{6n}{4n-1}\sum_{s=1}^3\tau_{s}(I_s\gr,\gr)+\frac{4n+2}{4n-1}\sum_{s=1}^3T(\xi_s,I_s\gr,\gr)\Big]\vol.
\end{multline}
Applying the identities from  Theorem~\ref{sixtyseven} and \eqref{propt} we calculate
\begin{equation}\label{tau}
\sum_{s=1}^3\tau_{s}(I_sX,Y)=\frac{n+2}nT^0(X,Y) +3Sg(X,Y).
\end{equation}
Using the first equality in Theorem~\ref{sixtyseven}, \eqref{boh4} and \eqref{tau}, we express the second line in \eqref{bohin4} in terms of $Ric, T^0$ and $U$ as follows
\begin{multline}\label{rtu}
Ric(\nabla f,\nabla f) -\frac{6n}{4n-1}\sum_{s=1}^3\tau_{s}(I_s\gr,\gr)+\frac{4n+2}{4n-1}\sum_{s=1}^3T(\xi_s,I_s\gr,\gr)\\=
\frac{2(n-1)(2n+1)}{(4n-1)(n+2)}\Big[Ric(\gr,\gr)+ \alpha_n T^0(\gr,\gr)+ \beta_n U(\gr,\gr)\Big],
\end{multline}
where
\begin{equation}
\alpha_n=\frac {2(4n+5)}{2n+1}, \qquad \beta_n=3\frac{2n^2+5n-1}{(2n+1)(n-1)}. %\frac{16n^2+4n-20}{4n^2-2n-2} %\beta_n=\frac{24n^2+60n-12}{4n^2-2n-2}
\end{equation}
At this point we let $f$ be an eigenfunction of the sub-Laplacian with eigenvalue $\lambda$, i.e., \eqref{eig} holds. An integration by parts yields
\begin{equation}\label{parts}
\int_M(\triangle f)^2\,\vol=\lambda\int_Mf\triangle f\,\vol=\lambda\int_M|\gr|^2\,\vol.
\end{equation}
Let us assume  $n\ge 2$. A substitution of \eqref{parts} and \eqref{rtu} in \eqref{bohin4} gives
\begin{equation}\label{bohin5}
0 \ge \int_M -\lambda|\gr|^2+\frac{n}{n+2}\Big[ Ric(\gr,\gr)+\alpha_n T^0(\gr,\gr)
+ \beta_n U(\gr,\gr)\Big] \vol.
\end{equation}
The conditions of the theorem together with \eqref{bohin5} yield the inequality
\begin{equation}\label{inin}
0 \ge\int_M\left (-\lambda+\frac{n}{n+2}k_0\right )|\gr|^2\,\vol,
\end{equation}
which implies the desired inequality $$\lambda \ge \frac{n}{n+2}k_0.$$
This completes the proof of Theorem~\ref{main1}.
\end{proof}

\begin{rmrk}
Suppose we have the case of equality in Theorem~\ref{main1}, i.e., we have
\begin{equation*}\label{eq1}
\lambda = \frac{n}{n+2}k_0, \qquad \triangle f=\frac{n}{n+2}k_0 f
\end{equation*}
For $c$ given by \eqref{c} equalities in \eqref{coshy1} and \eqref{coshy3} must hold which implies that the horizontal Hessian of the eigenfunction $f$ is given by the next  equation
\begin{equation}\label{eq7}
\bi^2f(X,Y)=-\frac{k_0}{4(n+2)}fg(X,Y)-\sum_{s=1}^3df(\xi_s)\omega_s(X,Y).
\end{equation}
\end{rmrk}

\section{Proof of Theorem~\ref{main0}}
We proof Theorem~\ref{main0} using the Lichnerowicz' estimate for the first positive eigenvalue of Riemannian Laplacian and the Obata' theorem \cite{O3} which says that the equality in the Lichnerowicz' estimate is achieved only on the round sphere.
\subsection{Relation between the Laplacian and the sub-Laplacian}
We start with the next lemma relating the Riemannian Laplacian and the sub-Laplacian.
\begin{lemma}\label{lapl}
Let $M$ be a $(4n+3)$-dimensional qc manifold. Then the sub-Laplacian $\triangle$ and the Riemannian Laplacian $\triangle^g$, corresponding to the Levi-Civita connection $\nabla^g$ of the extended metric $g$, are connected by
\begin{equation}\label{eq18}
\triangle^g f=\triangle f-\sum_{s=1}^3\xi_s^2f+df(\sum_{s=1}^3\nabla_{\xi_s}\xi_s).
\end{equation}
\end{lemma}
\begin{proof}
By definition, $\triangle^gf=-\sum_{a=1}^{4n}\nabla^gdf(e_a,e_a)-\sum_{s=1}^3\nabla^gdf(\xi_s,\xi_s)$,
where  $\{e_1,\ldots,e_{4n},\xi_1,\xi_2,\xi_3\}$ is an orthonormal basis of $H\oplus V$. Write $\widetilde{df}$ for the gradient of $f$, the last equality  can be write in the form
\begin{equation}\label{eq19}
\triangle^gf=-g(\nabla_{e_a}^g\widetilde{df},e_a)-\sum_{s=1}^3g(\nabla_{\xi_s}^g\widetilde{df},\xi_s)=
-g(\bi_{e_a}\widetilde{df},e_a)-\sum_{s=1}^3g(\nabla_{\xi_s}\widetilde{df},\xi_s),
\end{equation}
where we used \eqref{lcbi} and the identities
\begin{equation}\label{idt}
T(e_a,A,e_a)=T(\xi_s,A,\xi_s)=0
\end{equation}
following from the properties of the torsion tensor $T$ of $\nabla$ listed in \eqref{sixtyfour}. Now, we get \eqref{eq18} from \eqref{eq19}.
\end{proof}

Next we give an estimate between the first eigenvalues of the  Riemannian Laplacian and the sub-Laplacian.
\begin{prop}
Let $M$ be a $(4n+3)$-dimensional closed compact qc manifold.  The first positive eigenvalue $\mu$ of the Riemannian Laplacian and the first positive eigenvalue $\lambda$ of the sub-Laplacian satisfy the following inequality
\begin{equation}\label{eigenv}
\mu\leq\lambda+\int_M\sum_{s=1}^3(df(\xi_s))^2 \vol
\end{equation}
for any smooth function $f$ with $\int_Mf^2\,\vol=1$.
\end{prop}
\begin{proof}
From the variational characterization of the first eigenvalue and \eqref{eq18} we have the estimate
\begin{equation}\label{est}
\mu\leq\int_M(\triangle^gf)f\,\vol=\int_M(\triangle f)f\,\vol-\int_M\Big[\sum_{s=1}^3(\xi_s^2f)f-df(\sum_{s=1}^3\nabla_{\xi_s}\xi_s)f\Big]\,\vol.
\end{equation}
For the term $df(\sum_{s=1}^3\nabla_{\xi_s}\xi_s)$, we obtain consecutively
\begin{equation}\label{eq20}
df(\sum_{s=1}^3\nabla_{\xi_s}\xi_s)=g(\widetilde{df},\sum_{s=1}^{3}\nabla_{\xi_s}\xi_s)=
\sum_{t=1}^3df(\xi_t)g(\xi_t,\sum_{s=1}^{3}\nabla_{\xi_s}\xi_s)=\sum_{s,t=1}^{3}df(\xi_{t})g(\nabla_{\xi_s}\xi_s,\xi_t),
\end{equation}
where we used for the third equality that the Biquard connection is metric.
%Therefore, by changing the sumation indices and using the second equality in \eqref{xider}, we obtain
%\begin{equation}\label{eq20}
%df(\sum_{s=1}^3\nabla_{\xi_s}\xi_s)=-\sum_{s=1}^3[\alpha_u(\xi_t)-\alpha_t(\xi_u)]\xi_sf,
%\end{equation}
%where the triple $(s,t,u)$ is cyclic permutation of $(1,2,3)$.

Consider the vector field $fdf(\xi_s)\xi_s$. We calculate its  Riemannian divergence $div[fdf(\xi_s)\xi_s]$ as follows
\begin{multline}\label{diver}
div[f(\xi_sf)\xi_s]=(df(\xi_s))^2+(\xi_s^2f)f+fdf(\xi_s)\Big[g(\bi^g_{e_a}\xi_s,e_a)+
\sum_{t=1}^3g(\bi^g_{\xi_t}\xi_s,\xi_t)\Big]\\=(df(\xi_s))^2+(\xi_s^2f)f+fdf(\xi_s)\Big[g(\bi_{e_a}\xi_s,e_a)+
\sum_{t=1}^3g(\bi_{\xi_t}\xi_s,\xi_t)\Big]\\=(df(\xi_s))^2+(\xi_s^2f)f-
fdf(\xi_s)\sum_{t=1}^3g(\bi_{\xi_t}\xi_t,\xi_s)\Big],
\end{multline}
where we used \eqref{lcbi},  \eqref{idt} and the fact that the Biquard connection preserves the splitting $H\oplus V$ to established the second and the third equality. A substitution of \eqref{diver} and \eqref{eq20} in \eqref{est} followed by an application of the Riemannian divergence formula give inequality \eqref{eigenv}.
\end{proof}
\begin{proof}[Proof of Theorem~\ref{main0}]
Suppose that $M$ is a qc-Einstein structure of dimension at least eleven with a normalized qc scalar $S=2$, hence the qc Ricci tensor given by the first equality in \eqref{sixtyfour} satisfies $Ric=4(n+2)g$. Suppose the equality case of Theorem~\ref{main0} holds, i.e., $\lambda=4n$ and let $\lap f = \lambda f$. After a possible rescaling of $f$ and using the divergence formula we have then the following identities
\begin{equation}\label{eq4}%\label{eq17}
\begin{aligned}
& \lambda=4n,\qquad \triangle f= 4nf,\qquad\int_Mf^2\vol=1,\\
& \int_M |\gr|^2 \vol= \lambda=\frac {1}{\lambda}\int_M |\lap f|^2\vol.
\end{aligned}
\end{equation}
In this case Lemmas \ref{gr1} and \ref{gr2} together with equation \eqref{tau} yield
\begin{equation}\label{eq111}
\int_M\sum_{s=1}^3(df(\xi_s))^2\vol=3.
\end{equation}
Therefore, from \eqref{eq111}  we have the inequality
  \begin{equation}\label{Lich1}
  \mu\leq 4n+3.
  \end{equation}
On the other hand, any qc-Einstein manifold with a positive qc scalar curvature is locally 3-Sasakian \cite{IMV} and it is well known that a 3-Sasakian manifold is Einstein (with respect to the extended metric) with Riemannian scalar curvature (4n+2) \cite{Kas}, i.e., the Riemannian Ricci tensor $Ric^g$ is given by
 \begin{equation}\label{eq21}
 Ric^g(A,A)=(4n+2)g(A,A).
 \end{equation}
 By Lichnerowicz' theorem  and \eqref{eq21} we have
 \begin{equation}\label{Lich}
 \mu\geq 4n+3.
 \end{equation}
 The inequalities \eqref{Lich1} and \eqref{Lich} yield the equality
 \begin{equation}\label{eq22}
 \mu=4n+3.
 \end{equation}
 Therefore, by Obata's result we conclude that the manifold $(M,g)$ is isometric to the sphere $S^{4n+3}(1)$ and hence the manifold $(M,g,\mathbb Q)$ is qc equivalent to the $3$-Sasakian sphere of dimension $4n+3$. This completes the proof of Theorem~\ref{main0}.
\end{proof}

\section{Sharp estimates for square integrals of derivatives}\label{s:cordes}
In this section we prove Theorem \ref{t:lapl-hess est}.
\begin{proof}[Proof of Theorem \ref{t:lapl-hess est}]
 Notice that we are using a function which vanishes outside some compact so the integrals are well defined. The proof is similar to the proof of Theorem \ref{main1} except we have to express $|(\bi^2 f)_{[3]}|^2$ in two different ways. This is the place where the qc case differs from the CR case. We start with the identity \eqref{bohin}, in which we first move the integral of the square of the sub-Laplacian to the left-hand side of the equality. Then we write $$ |(\bi^2 f)_{[3]}|^2=(1-c)\,|(\bi^2 f)_{[3]}|^2+c\, |(\bi^2 f)_{[3]}|^2$$ and use \eqref{coshy3} to obtain
\[
|(\bi^2 f)_{[3]}|^2\geq \frac {1-c}{4n}\,|\lap f|^2 + c\, |(\bi^2 f)_{[3]}|^2
\]
when $1-c\geq 0$. Finally, we use \eqref{1} for the last term in the thus obtained form of \eqref{bohin}. The result is the following inequality (valid for $1-c\geq 0$)
\begin{multline}\label{e:cordes1}
\left ( 1-\frac {1-c}{4n} \right )\int_M |\triangle f|^2\vol \geq \int_M \Big[\Big(c+\frac{3}{n}\Big)|(\bi^2f)_{[3]}|^2 + \Big(1-\frac{1}{n}\Big)|(\bi^2f)_{[-1]}|^2\Big]\vol\\
+\int_M \Big[Ric(\nabla f,\nabla f)
-2\sum_{s=1}^3\tau_{s}(I_s\gr,\gr)+ 2\sum_{s=1}^3T(\xi_s,I_s\gr,\gr)\Big] \vol.
\end{multline}
In order to obtain the norm of horizontal Hessian we solve for $c$ the equation $$c+\frac 3n = 1-\frac 1n,$$ which gives $c=(n-4)/n$. Since $1-c=4/n>0$ we let $c=(n-4)/n$ in the above inequality \eqref{e:cordes1} which becomes
\begin{multline}\label{e:cordes2}
\frac {n^2-1}{n^2}\int_M|\triangle f|^2\vol \geq \frac {n-1}{n}\int_M  |\bi^2 f|^2\vol\\
+\int_M \Big[Ric(\nabla f,\nabla f)
-2\sum_{s=1}^3\tau_{s}(I_s\gr,\gr)+ 2\sum_{s=1}^3T(\xi_s,I_s\gr,\gr)\Big] \vol.
\end{multline}
 Recalling the formula for the Ricci tensor in Theorem \ref{sixtyseven}, \eqref{boh4} and \eqref{tau} after a short simplification (using that $n>1$) we obtain the desired inequality, which completes the proof.
\end{proof}
For a qc-Einstein manifold, where $T^0=U=0$,  Theorem \ref{t:lapl-hess est} gives the next corollary taking into account that a qc-Einstein manifold of dimension eleven and higher is of constant scalar curvature, see \cite{IMV}.
\begin{cor}
Let $(M, \eta)$ be a $(4n+3)$-dimensional qc-Einstein manifold, $n>1$. For any $f\in \mathcal{C}_o^\infty (M)$ we have
\begin{equation}\label{e:sasakian lapl-hess est}
\int_M \vert \lap f \vert^2\vol \geq \frac {n}{n+1}\int_M \vert \bi^2 f\vert^2\vol
+ \frac {2n^2S}{n+1}\int_M |\gr|^2 \vol.
\end{equation}
\end{cor}
For the quaternionic Heisenberg group with its standard qc structure, see \cite{IMV} and \cite{IV2}, the above Corollary gives the following result. The point here is the precise value of the constant $c_n$ since even the more general Calder\'on-Zygmund  $L^p$ version is well known to hold on nilpotent Lie groups, see \cite{Fo77} for an excellent overview.
\begin{cor}\label{t:lapl-hess est on heis}
 Let $(\QH, \tilde\Theta)$ be the $(4n+3)$-dimensional Heisenberg group equipped with its standard qc structure.  For any $f\in \mathcal{C}_o^\infty (\QH)$ we have
\begin{equation}\label{e:lapl-hess est on heis}
\norm{ \bi^2 f }_{L^2(\QH)}\leq c_n\,\norm{ \lap f }_{L^2(\QH)} , \qquad c_n= \sqrt{1+\frac {1}{n}}.
\end{equation}
 %in the sense that there is $\Phi\in \domoH$ for which equality holds. Here %$\domoH$ is the closure of $ \mathcal{C}_o^\infty (\QH)$ in the norm defined with the help of the left-hand side of %\eqref{e:lapl-hess est on heis}.
\end{cor}
As a consequence of the above estimate, \cite{DoMa05} and \cite{Do08} which generalize Cordes' results to the sub-Riemannian setting it follows that for
\begin{equation}\label{e:p-lapl regula}
2\leq p < 2+\frac {n+n\sqrt{16n^2+8n-3}}{4n^2+2n-1}
\end{equation}
a p-harmonic function on an open set $\Omega\subset \QH$  on the quaternionic Heisenberg group of dimension $4n+3$, $f\in S^{1,p}(\QH)$, has in fact additional regularity $f\in S_{loc}^{2,2}(\QH)$. Here $S^{k,p}\, (\Om)$ denote the usual non-isotropic Sobolev spaces, see for example \cite{Fo77}.  Similarly to \cite{DoMa05} and \cite{Do08} one can then obtain  a $\mathcal{C}^{1,\alpha}$ under suitable restrictions on $p$. Obtaining the $\mathcal{C}^{1,\alpha}$ property of the solution is in general still an open problem except in some cases, see \cite{Do08}, \cite{MZ-GZh} and \cite{Gar09} and references therein. The first $\mathcal{C}^{1,\alpha}$ estimate was obtained for the-Laplacian operator on the Heisenberg group \cite{Ca97}.

\end{document}